\documentclass[12pt,A4,reqno]{amsart}
\usepackage{amsfonts}
\usepackage{mathrsfs}
\usepackage[T1]{fontenc}
\usepackage{amssymb}
\usepackage{amsmath,amscd}
\usepackage{color}
\usepackage{epsf}
\usepackage{graphicx}
\usepackage[curve]{xypic}

\theoremstyle{plain}
\newtheorem{thm}{Theorem}[section]
\newtheorem{lem}[thm]{Lemma}
\newtheorem{prop}[thm]{Proposition}
\newtheorem{cor}[thm]{Corollary}

\theoremstyle{definition}
\newtheorem{defi}[thm]{Definition}
\newtheorem{rem}[thm]{Remark}
\newtheorem{exam}{Example}

\newcommand{\R}{\mathbb R}
\newcommand{\Z}{\mathbb Z}

\newcommand{\nn}{\vskip 0.2cm}
\newcommand{\n}{\vskip 0.1cm}

\begin{document}

\title [\ ] {Small Covers, infra-solvmanifolds and curvature}

\author{Shintar\^{o} Kuroki$^\dag$}
\address{$^\dag$ Osaka City University Advanced Mathematical Institute (OCAMI),
      3-3-138 Sugimoto, Sumiyoshi-ku, Osaka 558-8585, Japan \newline \textit{and}
             Department of Mathematics, University of Toronto, Room 6290, 40 St.
     George Street, Toronto, Ontario, M5S 2E4, Canada}
 \email{kuroki@scisv.sci.osaka-cu.ac.jp, shintaro.kuroki@utoronto.ca}

\author{Mikiya Masuda$^\ddag$}
\address{$^\ddag$Department of Mathematics, Osaka City University, Sugimoto,
     Sumiyoshi-Ku, Osaka, 558-8585, Japan}
 \email{masuda@sci.osaka-cu.ac.jp}

\author{Li Yu*}
\address{*Department of Mathematics and IMS, Nanjing University, Nanjing, 210093, P.R.China
  }
 \email{yuli@nju.edu.cn}
 \thanks{2010 \textit{Mathematics Subject Classification}. 57N16, 57S17, 57S25,
 53C25, 51H30\\
 $^\dag$The author is partially supported by the JSPS Strategic Young Researcher Overseas Visits
  Program for Accelerating Brain Circulation
  ``Deepening and Evolution of Mathematics and Physics, Building of International Network Hub based on
  OCAMI''.\\
  $^\ddag$The author is partially supported by Grant-in-Aid for Scientific Research
  22540094.\\
 *The author is partially supported by JSPS (Grant No. P10018) and
 Natural Science Foundation of China (grant no.11001120). This work is also
 funded by the PAPD (priority academic program development) of Jiangsu higher education institutions.}


\keywords{Small cover, real moment-angle manifold, real Bott
  manifold, Ricci curvature, infra-solvmanifold, infra-nilmanifold, Coxeter group}


\begin{abstract}
  It is shown that a small cover (resp. real moment-angle manifold) over a simple polytope
  is an infra-solvmanifold if and only if it is diffeomorphic to a real Bott manifold (resp. flat torus).
  Moreover, we obtain several equivalent conditions for a small cover being homeomorphic to
  a real Bott manifold. In addition,
  we study Riemannian metrics on small
  covers and real moment-angle manifolds with certain conditions on the Ricci or sectional curvature.
   We will see that these curvature conditions put very strong
  restrictions on the topology of the corresponding small covers
  and real moment-angle manifolds and the combinatorial structures of
   the underlying simple polytopes.
 \end{abstract}

\maketitle

 \section{Introduction}
 The notion of small cover is first introduced by
  Davis and Januszkiewicz~\cite{DaJan91} as
  an analogue of a smooth projective toric variety in the category of closed manifolds
  with $\Z_2$-torus actions.
 An $n$-dimensional \emph{small cover} $M^n$ is a closed $n$-manifold
  with a locally standard $(\Z_2)^n$ action whose orbit space can be identified with
  a simple convex polytope $P^n$ in the Euclidean space $\R^n$.
   The $(\Z_2)^n$-action on $M^n$ determines a $(\Z_2)^n$-valued \emph{characteristic
       function} $\lambda_{M^n}$ on the facets of
      $P^n$, which encodes the information of isotropy subgroups of
      the non-free orbits. Conversely, we can recover
      $M^n$ and the $(\Z_2)^n$-action, up to equivariant homeomorphism,
       by gluing $2^n$ copies of $P^n$ according to the function $\lambda_{M^n}$.
     It is shown in~\cite{DaJan91} that many important topological invariants of $M^n$
     can be easily computed in terms of the combinatorial
      structure of $P^n$ and the $\lambda_{M^n}$. For example,
      the fundamental group of $M^n$ is a finite index subgroup of a right-angled Coxeter group
      $W_{P^n}$, where $W_{P^n}$ is canonically determined by $P^n$. Note that not all simple convex
      polytopes admit small covers over them. But for any simple convex polytope $P^n$,
      we can canonically associate a closed manifold $\R\mathcal{Z}_{P^n}$
      to $P^n$ called a \emph{real moment-angle manifold} (see~\cite{DaJan91}
      and~\cite{BP02}).
      \nn

      In this paper, we will
      mainly use fundamental groups to study different kinds of geometric structures on small
      covers and real moment-angle manifolds.
     The following are some results proved in this paper.\nn

   \begin{thm} If the real moment-angle manifold
      $\R\mathcal{Z}_{P^n}$ of an $n$-dimensional simple convex polytope $P^n$
      is homeomorphic to an infra-solvmanifold, then $P^n$ is an $n$-cube
      and $\R\mathcal{Z}_{P^n}$ is diffeomorphic to the $n$-dimensional flat torus.
     If a small cover is homeomorphic to an infra-solvmanifold, it must be
       diffeomorphic to a real Bott manifold $\mathrm{(see\ Corollary~\ref{Cor:Infra-Solv-2})}$.
 \end{thm}

  \begin{thm}
   Let $\R\mathcal{Z}_{P^n}$ be the real moment-angle manifold of an $n$-dimensional
   simple convex polytope $P^n$.
  \begin{itemize}
    \item[(i)] $\R\mathcal{Z}_{P^n}$ admits a Riemannian metric with positive
    constant sectional curvature if and only if $P^n$ is an
    n-simplex $\mathrm{(see\ Theorem~\ref{thm:Spherical})}$.\n

    \item[(ii)]  $\R\mathcal{Z}_{P^n}$ admits a flat Riemannian
   metric if and only if $P^n$ is an n-cube $\mathrm{(see\ Corollary~\ref{Cor:Infra-Solv-2})}$.\n

   \item[(iii)] If $\R\mathcal{Z}_{P^n}$ admits a Riemannian
    metric with negative (not necessarily constant) sectional curvature, then
    no $2$-face of $P^n$ can be a $3$-gon or a $4$-gon
    $\mathrm{(see\ Proposition~\ref{Prop:Negative-Curv})}$.
   \end{itemize}
  \end{thm}

        The paper is organized as follows. In section 2, we
     study when the fundamental group of a real moment-angle
     manifold (or a small cover) over a simple convex polytope is virtually nilpotent or finite
     (Corollary~\ref{Cor:virtually-nilpotent} and Corollary~\ref{Cor:Finite-Fund-Group}).
     Then we introduce a special class of small covers called
     \emph{generalized real Bott manifolds} which are the main examples
     related to our study in this paper.
     In section 3, we study when a small cover or real moment-angle manifold
     is an infra-solvmanifold. It turns out that such a small cover (or real moment-angle manifold)
     must be a real Bott manifold (or flat torus) (Corollary~\ref{Cor:Infra-Solv-2}).
     The proof essentially uses a nice result in~\cite{DavJanScott98} that describes the asphericality of
      small covers in terms of flagness of the underlying simple polytope.
      In addition, we will briefly discuss the
       smooth structures on small covers and real moment-angle
      manifolds in the middle.
      In section 4, we obtain several equivalent conditions for a small cover to be
       homeomorphic to
     a real Bott manifold (Theorem~\ref{thm:Test-Real-Bott}).
     In section 5, we study Riemannian metrics on small
     covers and real moment-angle manifolds with various conditions on the Ricci or sectional curvature.
     Most of the results obtained in this section follow from the study of fundamental groups in section 2.
     In addition, some problems are proposed for future study.
     \\

   \section{Right-angled Coxeter group and fundamental group}

     Suppose $P^n$ is an $n$-dimensional simple convex polytope in the Euclidean space $\R^n$.
     Here the word ``simple'' means that any vertex of $P^n$ is the
     intersection of exactly $n$ different facets of $P^n$.
     Let $\mathcal{F}(P^n)$ denote the set of all facets of $P^n$.
      Let $W_{P^n}$ be a right-angled Coxeter group with one generator for each facet of
     $P^n$ and relations $s^2=1$, $\forall\,s\in \mathcal{F}(P^n)$, and $(st)^2 =1$ whenever
      $s,t$ are adjacent facets of $P^n$.  \nn

    \textbf{Remark:} Although we call $W_{P^n}$ a right-angled Coxeter
      group, the dihedral angle between two adjacent facets of $P^n$ may not be a right angle in reality. So
   generally speaking,
   $W_{P^n}$ is not the group generated by the reflections of $\R^n$
   about the hyperplanes passing the facets of $P^n$.\nn

  Suppose $F_1,\cdots, F_r$ are all the facets of $P^n$. Let $e_1,\cdots, e_r$ be a basis of
  $(\Z_2)^r$. Then we define a function $\lambda_0 : \mathcal{F}(P^n) \rightarrow (\Z_2)^r$ by
    \begin{equation}\label{equ:Lambda_0}
     \lambda_0(F_i) = e_i, \ 1\leq i \leq r.
    \end{equation}
      For any proper face $f$ of $P^n$, let $G_f$ denote the subgroup of $(\Z_2)^r$ generated by
  the set $\{ \lambda_0(F_i) \, |\, f\subset F_i \}$. For any point $p\in
  P^n$, let $f(p)$ denote the unique face of $P^n$ that contains $p$ in
  its relative interior. In~\cite[Construction 4.1]{DaJan91},
 the \emph{real moment-angle manifold} $\R \mathcal{Z}_{P^n}$ of $P^n$ is
  defined to be the following quotient space
    \begin{equation} \label{Equ:Real-Moment-Angle}
     \R \mathcal{Z}_{P^n} := P^n\times (\Z_2)^r \slash \sim
    \end{equation}
  where $(p,g) \sim (p',g')$ if and only if $p=p'$ and $g^{-1}g' \in G_{f(p)}$.
       It is shown in~\cite{DaJan91} that the fundamental group of
        $\R \mathcal{Z}_{P^n}$ is the kernel of
        the abelianization $\text{Ab}\colon W_{P^n}\to W^{ab}_{P^n}\cong(\Z_2)^r$,
        that is, there is an exact sequence
    \begin{equation} \label{Equ:Seq-2}
     1 \longrightarrow  \pi_1(\R \mathcal{Z}_{P^n})  \longrightarrow W_{P^n}
      \overset{\text{Ab}}{\longrightarrow} (\Z_2)^r \longrightarrow 1,
      \end{equation}
  so we have
  \[ \pi_1(\R \mathcal{Z}_{P^n}) = \mathrm{ker}(\text{Ab})
    = [W_{P^n},W_{P^n}] \ \text{(the commutator subgroup of $W_{P^n}$)}.  \]

    A \emph{small cover} $M^n$ over $P^n$ is a closed $n$-manifold with a locally standard
    $(\Z_2)^n$-action so that its orbit space is homeomorphic to
    $P^n$. Here ``locally standard'' means that any point in $M^n$ has a
    $(\Z_2)^n$-invariant open neighborhood which is equivariantly
    homeomorphic to a $(\Z_2)^n$-invariant open subset in an
    $n$-dimensional faithful linear representation space of $(\Z_2)^n$.
    Let $\pi: M^n \rightarrow P^n$ be the quotient map.
    For any facet $F_i$ of $P^n$, the isotropy subgroup of
    $\pi^{-1}(F_i)$ in $M^n$ under the $(\Z_2)^n$-action is
    a rank one subgroup of $(\Z_2)^n$ generated by a nonzero element, say $g_{F_i} \in (\Z_2)^n$.
    Then we obtain a map $\lambda_{M^n}: \mathcal{F}(P^n) \rightarrow
    (\Z_2)^n$ where
     $$ \lambda_{M^n}(F_i) = g_{F_i}, \ 1\leq i \leq r. $$
     We call $\lambda_{M^n}$ the \emph{characteristic function} associated
     to $M^n$. It is shown in~\cite{DaJan91} that
     up to equivariant homeomorphism, $M^n$
     can be recovered from $(P^n,\lambda_{M^n})$ in a similar way
     to the construction of $\R \mathcal{Z}_{P^n}$
     in~\eqref{Equ:Real-Moment-Angle}, that is
     \begin{equation} \label{Equ:Small-Cover}
        M^n = P^n\times (\Z_2)^n \slash \sim
     \end{equation}
      where $(p,g) \sim (p',g')$ if and only if $p=p'$ and $g^{-1}g' \in G^{\lambda_{M^n}}_{f(p)}=$
      the subgroup of $(\Z_2)^n$ generated by $\{ \lambda_{M^n}(F_i) \, |\, f(p)\subset F_i
      \}$.\nn

     Moreover, $\lambda_{M^n}$ determines a group homomorphism
     $\overline{\lambda}_{M^n} : (\Z_2)^r \rightarrow (\Z_2)^n$
     where
      \[  \overline{\lambda}_{M^n}(e_i) = \lambda_{M^n} (F_i) = g_{F_i},  \ 1\leq i \leq r. \]
     It is shown~in \cite{DaJan91} that the fundamental group $\pi_1(M^n)$ of $M^n$
     is isomorphic to the kernel of the composition $\overline{\lambda}_{M^n}\circ \text{Ab}$,
     that is, there is an exact sequence
       \begin{equation} \label{Equ:Fund-Group}
      1 \longrightarrow  \pi_1(M^n)  \longrightarrow W_{P^n}
      \overset{\overline{\lambda}_{M^n}\circ \text{Ab}}{\longrightarrow} (\Z_2)^n \longrightarrow 1.
       \end{equation}
  Then it follows from \eqref{Equ:Seq-2} and \eqref{Equ:Fund-Group} that we have an exact sequence
   \begin{equation} \label{Equ:moment-small}
     1 \longrightarrow  \pi_1(\R \mathcal{Z}_{P^n})  \longrightarrow \pi_1(M^n)
      \longrightarrow (\Z_2)^{r-n} \longrightarrow 1.
      \end{equation}
 In fact, it is easy to see that $\R \mathcal{Z}_{P^n}$ is a regular $(\Z_2)^{r-n}$-covering of any small
 cover over $P^n$.\nn

  \begin{prop} \label{Prop:2-neighborly}
For an $n$-dimensional simple convex polytope $P^n$, the following are equivalent:
  \begin{enumerate}
    \item[(i)] $P^n$ is $2$-neighborly (i.e. any two facets of $P^n$ are adjacent).\n
    \item[(ii)] The real moment-angle manifold  $\R\mathcal{Z}_{P^n}$  over
               $P^n$ is simply connected. \n
  \end{enumerate}
 Moreover, if there exists a small cover $M^n$ over $P^n$, then the above statements
 are also equivalent to the following:
 \begin{enumerate}
     \item[(iii)]
      The fundamental group of $M^n$ is isomorphic to $(\Z_2)^{r-n}$, where $r$
    is the number of facets of $P^n$.\n
\end{enumerate}

 \end{prop}
 \begin{proof}
   By definition, $P^n$ is $2$-neighborly if and only if $W_{P^n}$ is isomorphic to $(\Z_2)^r$ where $r$ is the number
   of facets of $P^n$. Then the equivalence between (i) and (ii) follows from
   \eqref{Equ:Seq-2}, and the equivalence between (ii) and
 (iii) follows from \eqref{Equ:moment-small}.
    \end{proof}
 \nn

     All the relations between the generators of $W_{P^n}$ can be formally represented by a
      matrix $m=(m_{st})$, called \emph{Coxeter matrix}.
      \[ m_{st} := \left\{
           \begin{array}{ll}
             1, & \hbox{if $s=t$;} \\
             2, & \hbox{if $s$ is adjacent to $t$;} \\
             \infty, & \hbox{otherwise.}
           \end{array}
         \right. \  \text{($s,t$ denote arbitrary facets of $P^n$).} \]
     The triple $(W_{P^n}, \mathcal{F}(P^n), m)$ is called a \emph{Coxeter system} of $W_{P^n}$. \nn

   For a general Coxeter system $(W, S, m)$, its \emph{Coxeter graph} is a graph with a vertex
  set $S$, and with two vertices $s \neq t$ joined by an edge whenever
  $m_{st} \geq 3$. If $m_{st} \geq 4$, the corresponding edge is labeled by $m_{st}$.
  We say that $(W, S, m)$ is \emph{irreducible} if its Coxeter graph is
  connected.\nn

   A Coxeter group $W$ is called \emph{rigid} if, given any two systems $(W, S, m)$ and
   $(W, S',m')$ for $W$; there is an automorphism $\rho : W \rightarrow W$
      such that $\rho(S,m) = (S',m')$, i.e.
     the Coxeter graphs of $(W, S,m)$ and $(W, S,m')$ are isomorphic.\nn

  \begin{thm}[Radcliffe~\cite{Radc03}] \label{thm:Rigid}
       If $(W,S,m)$ is a Coxeter system with
        $m_{st} \in \{ 2, \infty \}$ for all $s\neq t \in S$, then W is rigid.
        In other words, any right-angled Coxeter group is rigid.
      \end{thm}

  Associated to any Coxeter system $(W, S, m)$, there is a symmetric bilinear
 form $(\, ,\, )$ on a real vector space $V$ with a basis $\{\alpha_s\, | \, s \in S\}$ in
 one-to-one correspondence with the elements of $S$. The bilinear form $(\, ,\, )$
 is defined by:
   \begin{equation} \label{Equ:Bilinear-Form}
    (\alpha_s, \alpha_t) := - \cos \frac{\pi}{m_{st}},
   \end{equation}
  where the value on the right-hand side is interpreted to be $-1$ when $m_{st} =
   \infty$.\nn

  It is well-known that a Coxeter group $W$ is finite if and only if
   the bilinear form of a Coxeter system $(W, S, m)$ is positive definite.
   All finite Coxeter groups have been classified by H.~S.~M.~Coxeter in 1930s
   (see~\cite{Coxeter35} and~\cite{GroBen85}).
  It is easy to see that if $W$ is a finite
   right-angled Coxeter group, $W$ must be isomorphic to
   $(\Z_2)^k$ for some $k\geq 0$. By Theorem~\ref{thm:Rigid}, any
    Coxeter graph of $W\cong (\Z_2)^k$ should be $k$ disjoint vertices. \nn

       An irreducible Coxeter group $W$ is called \emph{affine} if there is a
   Coxeter system $(W, S, m)$ so that the bilinear form of the system (see~\eqref{Equ:Bilinear-Form})
    is positive semi-definite but not positive definite.
   More generally, a Coxeter group is called \emph{affine} if its irreducible components are
   either finite or affine, and at least one component is affine.
   Equivalently, a Coxeter group is affine if it is an infinite group and
   has a representation as a discrete, properly acting reflection
   group in $\R^n$. The reader is referred to~\cite{Humph90}
    for more information on affine Coxeter groups.
    A Coxeter group is called \emph{non-affine} if it is not
   affine.\nn

    \begin{thm}[Qi~\cite{Qi07}] \label{thm:Qi}
    The center of any finite index subgroup of an infinite, irreducible,
    non-affine Coxeter group is trivial.
  \end{thm}

 A group is called \emph{virtually nilpotent} (\emph{abelian, solvable})  if it has
  a nilpotent (abelian, solvable) subgroup of finite index. \nn

  \begin{lem} \label{Lem:Vir-Nilp}
    If the
    right-angled Coxeter group $W_{P}$ of a simple convex polytope $P$ is virtually nilpotent,
    then $W_{P} \cong (\Z_2)^k \times (\widetilde{A}_1)^l$ for
    some $k,l\geq 0$,
   where $k+2l$ equals the number of facets of
   $P$ and $\widetilde{A}_1 = \langle a, b\, |\, a^2=1, b^2=1
   \rangle$.
  \end{lem}
  \begin{proof}
    Suppose $W_{P}$ is infinite and let $N$ be a finite index nilpotent subgroup of $W_{P}$.
   If $W_{P}$ is not affine, then $W_{P}$
    has at least one irreducible component which is neither finite nor affine, say
    $W_1$. Then $N_1 = N\cap W_1$ is a finite index nilpotent
    subgroup of $W_1$. Since $W_1$ is an infinite group, $N_1$ is also infinite hence nontrivial.
    Now since $W_1$ is infinite, irreducible and non-affine, by Theorem~\ref{thm:Qi},
     the center of $N_1$ must be trivial.
    But the center of any nontrivial nilpotent group is never trivial.
    This implies that the Coxeter group $W_{P}$ must be either finite or affine.\nn

    Since $W_{P}$ is right-angled, so is each of its irreducible components.
   Then by the classification of irreducible affine
   Coxeter group (see~\cite{Humph90}),
   each connected component of the Coxeter graph of $W_{P}$ is either a single vertex or
   a $1$-simplex labeled by $\infty$.
   The Coxeter group corresponding to a single vertex is $\Z_2$, and the Coxeter group corresponding to
   a $1$-simplex labeled by $\infty$ is
    $\widetilde{A}_1$.
     Suppose the Coxeter graph of $W_{P}$ consists of
     $k$ isolated vertices and $l$ isolated $1$-simplices
    labeled by $\infty$.  Then $W_{P} \cong (\Z_2)^k \times (\widetilde{A}_1)^l$
      and $k+2l$ equals the number of facets of $P$.
     \end{proof}

   \begin{cor} \label{Cor:virtually-nilpotent}
  If $\pi_1(\R\mathcal{Z}_P)$ is virtually nilpotent, then
  $\pi_1(\R\mathcal{Z}_P)\cong\Z^l$ for some $l\le r/2$ where $r$ is
 the number of facets of $P$. In particular, if
 $\pi_1(\R\mathcal{Z}_P)$ is finite, $\R\mathcal{Z}_P$ must be
 simply connected and so $P$ is $2$-neighborly.
\end{cor}
 \begin{proof}
     If $\pi_1(\R \mathcal{Z}_{P})$ is virtually nilpotent, so is the Coxeter group $W_{P}$.
      Then by Lemma~\ref{Lem:Vir-Nilp}, $W_{P} \cong (\Z_2)^k \times (\widetilde{A}_1)^l$
   for some $k,l\geq 0$ with $k+2l=r$. So $\pi_1(\R \mathcal{Z}_{P}) =[W_{P}, W_{P}] \cong \Z^l$ since
   $[\widetilde{A}_1, \widetilde{A}_1] \cong \Z$. If $\pi_1(\R\mathcal{Z}_P)$ is
   finite, then $l=0$, i.e. $\pi_1(\R\mathcal{Z}_P)$ is trivial.
  And so $P$ is $2$-neighborly by Proposition~\ref{Prop:2-neighborly}.
 \end{proof}

  \begin{cor} \label{Cor:Finite-Fund-Group}
   If a small cover $M^n$ over an $n$-dimensional simple convex
   polytope $P^n$ has finite fundamental group, then $P^n$ is
   $2$-neighborly and $\pi_1(M^n)$ is isomorphic to $(\Z_2)^{r-n}$, where $r$
    is the number of facets of $P^n$.
  \end{cor}
  \begin{proof}
   Because of~\eqref{Equ:moment-small}, $\pi_1(M^n)$ is finite if and only if
    $\pi_1(\R \mathcal{Z}_{P^n})$ is finite. So the claim follows from
     Corollary~\ref{Cor:virtually-nilpotent} and
     Proposition~\ref{Prop:2-neighborly}.
  \end{proof}
  \n

    \begin{exam} \label{Exam:Bott}
    Let $\Delta^j$ denote a $j$-simplex.
     If $P^n$ is a product of simplices $\Delta^{n_1}\times \cdots \times \Delta^{n_m}$,
     then the number of facets of $P^n$ is $n+m$ and $P^n$ is $2$-neighborly
     if and only if $n_i \geq 2$ for all $1\leq i \leq m$.
     The real moment-angle manifold $\R\mathcal{Z}_{P^n}$ is
     a product of spheres $S^{n_1}\times \cdots \times S^{n_m}$,
     so $\pi_1(\R\mathcal{Z}_{P^n})\cong \Z^l$ where $l$ is the number of $i$'s with $n_i=1$.  \n

         Small covers over $\Delta^{n_1}\times \cdots \times \Delta^{n_m}$ naturally arise as follows.
            Recall that a \emph{generalized real Bott manifold} is
          the total space $B_m$ of an iterated fiber bundle:
   \begin{equation} \label{Equ:Bm}
      B_m \overset{\pi_m}{\longrightarrow} B_{m-1} \overset{\pi_{m-1}}{\longrightarrow}
    \cdots \overset{\pi_2}{\longrightarrow} B_1 \overset{\pi_1}{\longrightarrow} B_0
   =\ \{ \text{a point} \} ,
   \end{equation}
   where each $B_i$ ($1\leq i \leq m$) is the projectivization of
   the Whitney sum of a finite collection (at least two) of real line bundles over $B_{i-1}$.
   We call the sequence in~\eqref{Equ:Bm} a \emph{generalized real Bott tower}.
   We consider $B_m$ as a closed smooth manifold whose smooth
        structure is determined by the bundle structures of $\pi_i: B_i \rightarrow
        B_{i-1}$, $i=1,\cdots, m$.
   Suppose the fiber of $\pi_i: B_i \rightarrow B_{i-1}$ is a real projective
   space of dimension $n_i$.
   Then it is easy to show that $B_m$ is a small cover over a product of simplices
   $\Delta^{n_1}\times \cdots \times\Delta^{n_m}$.
   Conversely,~\cite[Remark 6.5]{SuyMasDong10} tells us that any small cover
   over a product of simplices is homeomorphic to a generalized real Bott manifold.
    When $n_1=\cdots = n_m =1$, we call $B_m$ a \emph{real Bott manifold}
   (see~\cite{MasKam09-1}), which is a small cover over an $m$-cube.
       \end{exam}

\begin{rem}
    The number of (weakly) equivariant homeomorphism types of real Bott manifolds was
   counted in~\cite{CaiChenLu07} and~\cite{Choi08}. Moreover, it was shown in~\cite{MasKam09-1}
   and~\cite{SuMasOum10} that
   the diffeomorphism types of real Bott manifolds are completely determined by their cohomology rings
   with $\Z_2$-coefficients.
   This is called \emph{cohomological rigidity} of real Bott manifold.
   This property relates the diffeomorphism classification
  of real Bott manifolds with the classification of acyclic digraphs
 (directed graphs with no directed cycles) up to some equivalence
 (see~\cite{SuMasOum10}). But cohomological rigidity does not hold for
 generalized real Bott manifolds. Indeed,
 it is shown in~\cite{Mas-height-2} that there
    exist two generalized real Bott manifolds whose cohomology rings with $\Z_2$-coefficients
    are isomorphic, but they are not even homotopy equivalent.
\end{rem}

    \begin{rem}
      There are many $2$-neighborly simple convex polytopes which are not product of
      simplices. For example, the dual $P^n$ of an
       $n$-dimensional cyclic polytope ($n\geq 4$) is always $2$-neighborly.
       Recall that a \emph{cyclic polytope} $C(k,n)$ ($k > n$) is the convex hull of $k$
       distinct points on the curve $\gamma(t) =(t, t^2, \cdots, t^n)$ in $\R^n$.
        Then the number of facets of its dual $P^n$ has $k$ facets.
        But when $k\geq 2^n$, $P^n$
        admits no small cover hence can not be a product of simplices (see p.428 of~\cite{DaJan91}).
         It is an interesting problem to
      find out all the $2$-neighborly simple convex polytopes that admit small covers.
      Note that this problem is intimately related to the \emph{Buchstaber invariant} of simple polytopes
      (see~\cite[section 7.5]{BP02} and~\cite{Erokh08}).\\
   \end{rem}

   \section{Small cover, real moment-angle manifold and infra-solvmanifold}

It is shown in~\cite{MasKam09-1} that any $n$-dimensional real Bott
manifold
   admits a flat Riemannian metric which is invariant under the $(\Z_2)^n$-action.
Conversely, any small cover of dimension $n$
    which admits a flat Riemannian metric invariant under the
    canonical $(\Z_2)^n$-action must be a real Bott manifold.
    In fact, this is proved for any real toric manifolds in~\cite[Theorem 1.2]{MasKam09-1},
     but the same argument works for small covers.
    This suggests us to ask the following
   question. \nn

   \noindent \textbf{Question:}
    If a small cover $M^n$ of dimension $n$ admits a flat Riemannian metric (not necessarily invariant
     under the $(\Z_2)^n$-action), must $M^n$ be diffeomorphic to a real Bott
     manifold? Or equivalently, must $M^n$ be a small cover over an $n$-cube?
   \nn

   We will see that the answer to this question is yes (Corollary~\ref{Cor:Infra-Solv-2}).
    In fact, we
  will obtain a much stronger result in Corollary~\ref{Cor:Infra-Solv-2}.
   But first let us introduce a well-known notion in combinatorics.\nn

 \begin{defi}[Flag Complex]
   A simplicial complex $K$ is called \emph{flag} if a subset $J$ of the vertex set of $K$
   spans a simplex in $K$ whenever any two vertices in $J$ are joined by a $1$-simplex in $K$.
 \end{defi}
  Let $P$ be a simple convex polytope of dimension $n$.  For simplicity, we
 say that $P$ is \emph{flag} if the boundary of $P$ is dual to a flag
 complex (i.e. a collection of facets of $P$ have a common intersection
  whenever any two of them intersect). Suppose $ F_1,\cdots, F_r$ are all the
  facets of $P$.
  Then each $F_k$ itself is an $(n-1)$-dimensional simple convex polytope whose facets
   are $\{ F_{ki} := F_k\cap F_i \neq \emptyset, 1\leq i \leq r \}$.
  Let $s(F_k)$ denote the generator of $W_P$ corresponding to the facet
  $F_k$ ($1\leq k \leq r$) of $P$. Similarly, let $s(F_{ki})$ denote the
  the generator of $W_{F_k}$ corresponding to a facet $F_{ki}$ of $F_k$.

   \begin{lem} \label{Lem:Face-Flag}
     Suppose a simple convex polytope $P$ is flag and
     $ F_1,\cdots, F_r$ are all the facets of $P$. If $F_k\cap F_i\not=\emptyset$ for
   $i=i_1,\dots,i_q$, then $F_{ki_1}\cap\dots\cap F_{ki_q}\not=\emptyset$
    if and only if $F_{i_1}\cap\dots\cap F_{i_q}\not=\emptyset$.
    So if $P$ is flag, every facet of $P$ must also be flag.
 \end{lem}
 \begin{proof}
   The ``only if part'' is trivial.  Suppose $F_{i_1}\cap\dots\cap
    F_{i_q}\not=\emptyset$.  Then any two of $F_k, F_{i_1},\dots,
   F_{i_q}$ intersect, so the whole intersection $F_k \cap
   F_{i_1}\cap\dots\cap F_{i_q}$ is non-empty because
  $P$ is flag, proving the ``if part''.
\end{proof}

       \nn

  \begin{prop} \label{prop:flag-polytop-solv}
   Let $P^n$ be a simple convex polytope of dimension $n$.
   If $P^n$ is flag and the
    fundamental group of the real moment-angle manifold $\R \mathcal{Z}_{P^n}$ of $P^n$ is
    virtually solvable, then $P^n$ must be an $n$-cube.
 \end{prop}

\begin{proof}
  When $n=2$, $P^n$ is a polygon. Then our assumption that $P^n$ is flag and
  $\pi_1(\R \mathcal{Z}_{P^n})$ is virtually solvable will force
  $P^n$ to be a $4$-gon (see~\cite[Example 6.40]{BP02}).
  In the rest, we assume $n\ge 3$.
 By~\eqref{Equ:Seq-2}, we have an exact sequence
\begin{equation*}
  1\longrightarrow \pi_1(\R \mathcal{Z}_{P^n})\longrightarrow
  W_{P^n}\longrightarrow (\Z_2)^r \longrightarrow 1.
\end{equation*}
   Let $F_k$ $(k=1,\dots,r)$ be the facets of $P^n$. Then similarly, we have
 an exact sequence for each $F_k$
   \begin{equation*}
   1\longrightarrow \pi_1(\R\mathcal{Z}_{F_k})\longrightarrow
   W_{F_k}\longrightarrow (\Z_2)^{r_k}\longrightarrow 1
\end{equation*}
 where $r_k$ is the number of facets of $F_k$. Notice that $W_{F_k}$ is
 generated by $s(F_{ki})$ with relations $s(F_{ki})^2=1$ and
  $(s(F_{ki})s(F_{kj}))^2=1$ whenever $F_{ki}$ and
  $F_{kj}$ intersect. But $F_{ki}$ and $F_{kj}$ intersect if and
  only if $F_i$ and $F_j$ (having non-empty intersection with $F_k$)
  intersect by Lemma~\ref{Lem:Face-Flag}. This implies that the group homomorphism
 $\varphi : W_{F_k} \rightarrow W_{P^n}$ sending each $s(F_{ki})$ to $s(F_i)$ is
  injective. Moreover, since
   $$\pi_1(\R \mathcal{Z}_{P^n}) = [W_{P^n}, W_{P^n}], \ \ \pi_1(\R\mathcal{Z}_{F_k}) = [W_{F_k},
   W_{F_k}],$$
  $\varphi$ maps $\pi_1(\R\mathcal{Z}_{F_k})$
   injectively into $\pi_1(\R\mathcal{Z}_{P^n})$.
   Then $\pi_1(\R\mathcal{Z}_{F_k})$ is virtually solvable 
   since so is $\pi_1(\R\mathcal{Z}_{P^n})$
   by our assumption.
  In addition, $F_k$ is also flag by Lemma~\ref{Lem:Face-Flag}. \nn

  By iterating the above arguments, we can show that for any
  $2$-face $f$ of $P^n$, $f$ is flag and
  the fundamental group of the real moment-angle manifold $\R \mathcal{Z}_{f}$ is virtually
  solvable. We have shown that such an $f$ must be a $4$-gon.
  So any $2$-face of $P^n$ is a $4$-gon, which implies that $P^n$ is an $n$-cube
   (see~\cite[Problem 0.1]{Ziegler95}).
\end{proof}

    It is shown in~\cite[Theorem 2.2.5]{DavJanScott98}
   that a small cover over a simple convex polytope $P$ is aspherical if and only if $P$ is flag.
   Similarly, we can prove the following.\nn

  \begin{prop} \label{Prop:Asph-Flag}
   The real moment-angle manifold $\R \mathcal{Z}_P$ of a simple convex polytope $P$
   is aspherical if and only if $P$ is flag.
   \end{prop}
   \begin{proof}
    Define $\mathcal{M} =  P \times W_P \slash \sim$, where the
    equivalence relation is defined by $(x,w)\sim (x', w')$ $\Longleftrightarrow$
    $x'=x$ and $w'w^{-1}$ belongs to the subgroup $G_x$ of $W_P$ generated by
    $\{ s(F)\, ;\, x\in F\}$. If $x$ lies in the relative interior of a codimesion-$k$
    face of $P$, then the subgroup $G_x$ of $W_P$ is isomorphic to $(\Z_2)^k$.\nn

    It is shown in~\cite[Lemma 4.4]{DaJan91} that
    $\mathcal{M}$ is simply connected. Let $\zeta:  P \times W_P \rightarrow \mathcal{M} $ be the quotient map.
    There is a natural action of $W_P$ on $\mathcal{M}$ defined by:
    \begin{equation}
       w'\cdot \zeta(x,w) = \zeta(x,w'w), \ \, w,w'\in W_P,\ x\in P.
    \end{equation}
    The isotropy group of a point $\zeta(w,x) \in \mathcal{M}$ under this $W_P$-action is exactly
    $G_x$.\nn

    \textbf{Claim-1:} The commutator subgroup $[W_P, W_P]$
    of $W_P$ acts freely on $\mathcal{M}$.\nn

   It amounts to prove that $[W_P, W_P] \cap G_x = \{ 1 \}$ for any point $x\in P$. In fact,
   it is easy to see that the abelianization $\mathrm{Ab}: W_P \rightarrow W^{ab}_P$ maps
   $G_x$ injectively into $W^{ab}_P$. So $G_x\cap \text{ker}(\mathrm{Ab})=\{ 1 \}$, proving Claim-1. \nn

    \textbf{Claim-2:} The quotient space $\mathcal{M}\slash [W_P, W_P]$ is homeomorphic to
    $\R \mathcal{Z}_P$.\nn

    Suppose $F_1,\cdots, F_r$ are all the facets of $P$. For each $1\leq i\leq
    r$, let $\overline{s}(F_i)$ be the image of $s(F_i)$ under the
    abelianization $\mathrm{Ab}: W_P \rightarrow W^{ab}_P$.
    Then $\{ \overline{s}(F_1),\ldots, \overline{s}(F_r) \}$ is a basis of $W^{ab}_{P} \cong
    (\Z_2)^r$.
     So the quotient $\mathcal{M}\slash [W_P, W_P]$ is homeomorphic to
     the space obtained
     by gluing $2^r$ copies of $P$
     according to the characteristic function $\mu$ on $P$ where $\mu(F_i) = \overline{s}(F_i) \in (\Z_2)^r$,
     $1\leq i \leq r$. This coincides with the definition of
     $\R\mathcal{Z}_P$ (see~\eqref{Equ:Real-Moment-Angle}). So the
     Claim-2 is proved.\nn

     Therefore, $\mathcal{M}$ is a universal covering of
       $\R\mathcal{Z}_P$. The relation between $\mathcal{M}$ and $\R
       \mathcal{Z}_P$ is demonstrated in the following diagram.\n

       \[ \xymatrix{
           [W_P,W_P] \ar@/^3ex/@{->}[r]  & P \times W_P \ar[r]\ar[d] &  P \times (\Z_2)^r \ar[d]\\
        \pi_1(\R \mathcal{Z}_P) \ar@{=}[u] \ar@/^3ex/@{->}[r]^{\text{acts freely}} & P \times W_P\slash \sim \ar[r] \ar@{=}[d] &  P \times (\Z_2)^r\slash \sim
           \ar@{=}[d] \\
                 &     \mathcal{M}   & \R \mathcal{Z}_P
           } \]
           \nn

       So $\R \mathcal{Z}_P$ is aspherical
       if and only if $\mathcal{M}$ is contractible.
      But it is shown in~\cite[Theorem 2.2.5]{DavJanScott98} that
      $\mathcal{M}$ is contractible if and only if $P$ is flag. So
      the proposition is proved.
   \end{proof}
   \n
   The following corollary is an immediate consequence of
     Proposition~\ref{prop:flag-polytop-solv} and
     Proposition~\ref{Prop:Asph-Flag}.\nn

  \begin{cor}\label{Cor:Infra-Solv}
 Let $P^n$ be a simple convex polytope of dimension $n$.
  If the real moment-angle manifold $\R\mathcal{Z}_{P^n}$ (or a small cover) of $P^n$
  is aspherical with virtually solvable fundamental group, then $P^n$ is an $n$-cube.
 \end{cor}

  The typical examples of aspherical manifolds with virtually solvable fundamental groups
  are infra-solvmanifolds.  In fact, any compact aspherical manifold with virtually solvable fundamental
  group is homeomorphic to an infra-solvmanifold (see~\cite[Corollary 2.21]{FarJones90}). But
  in general, we can not replace the ``homeomorphic'' by ``diffeomorphic'' in this statement.
 On the other hand, compact infra-solvmanifolds are \emph{smoothly rigid}, i.e. any two
    compact infra-solvmanifolds with isomorphic fundamental groups
    are diffeomorphic (see~\cite[Theorem 2]{Wilk00} or~\cite[Corollary 1.5]{Bau04}).

   \nn

  \begin{defi}[Infrahomogeneous Space]\label{Def:Infrahomoge}
   Let $G$ be a connected and simply connected Lie group, $K$ be a maximal compact
 subgroup of the group $\mathrm{Aut}(G)$ of automorphisms of $G$, and $\Gamma$ be a
  cocompact, discrete subgroup of $E(G) = G \rtimes K$. If the
  action of $\Gamma$ on $G$ is free and $[\Gamma: G\cap \Gamma ] < \infty$,
  the orbit space $\Gamma \backslash G$ is called a \emph{compact
  infrahomogeneous space modeled on $G$}.
   If $G$ is solvable (nilpotent), $\Gamma \backslash G$ is called a
   \emph{compact infra-solvmanifold} (\emph{infra-nilmanifold}).
   When $G = \R^n$ and $K=O(n,\R)$ (the orthogonal group), $\Gamma \backslash G$
   is a compact \emph{flat Riemannian manifold}.
  \end{defi}

  In the above definition, the group law of $G \rtimes K < G\rtimes \mathrm{Aut}(G)$ is
  defined by:
  \[ (g_1,\tau_1) \cdot (g_2, \tau_2) = (g_1\cdot \tau_1(g_2), \tau_1\circ\tau_2), \ \, g_1,g_2\in G,\
  \tau_1,\tau_2\in \mathrm{Aut}(G). \]
  The action of $G\rtimes \mathrm{Aut}(G)$ on $G$ is defined by:
   \[  (g,\tau) \cdot g' = g \cdot \tau(g'),\ \, g, g'\in G,\ \tau\in \mathrm{Aut}(G).  \]

  By definition, we have the following hierarchy of notions:
  \begin{align*}
   & \ \quad \text{Compact flat Riemannian manifolds}
     \\
 &  \subset \text{Compact infra-nilmanifolds} \subset \text{Compact infra-solvmanifolds}\\
  &  \subset \text{Compact aspherical manifolds
  with virtually solvable fundamental
   groups}
 \end{align*}

 \begin{rem}
  There are at least three different versions of the definition of compact infra-solvmanifolds in
  the mathematical literature. 
  The one used here (Definition~\ref{Def:Infrahomoge}) is taken from~\cite{Tusch97}. The other two are:
   \begin{itemize}
     \item A \emph{compact infra-solvmanifold} is a manifold of the form $\Delta \backslash G$, where $G$ is a connected,
    simply connected solvable Lie group, and $\Delta$ is a torsion-free
  cocompact discrete subgroup of $\mathrm{Aff}(G)=G \rtimes \mathrm{Aut}(G)$
   which satisfies: the closure of $hol(\Delta)$ in $\mathrm{Aut}(G)$ is compact where
   $hol: \mathrm{Aff}(G) \rightarrow \mathrm{Aut}(G)$ is the holonomy projection
    (see~\cite[Definition 1.1]{Bau04}).\nn

     \item A \emph{compact infra-solvmanifold} is a double coset space $\Gamma \backslash G \slash K$
   where $G$ is a virtually connected and virtually solvable Lie
   group, $K$ is a maximal compact subgroup of $G$ and $\Gamma$ is a
   torsion-free, cocompact, discrete subgroup of $G$ (see~\cite[2.10]{FarJones90}).
   \end{itemize}

   These three definitions are actually equivalent (see~\cite{KuroYu-13} for explanation).
 \end{rem}

  Compact infra-nilmanifolds and infra-solvmanifolds also have some
  Riemannian geometric interpretations as follows.
   By a theorem of Ruh~\cite{Ruh82} which is based on the work of
   Gromov~\cite{Gromov78}, a compact connected smooth manifold $M$ is an infra-nilmanifold if and only if
    it is \emph{almost flat}, which means that
    $M$ admits a sequence of
  Riemannian metric $\{ g_{n}\}$ with uniformly bounded sectional curvature
  so that $(M, g_n)$ collapses in the \emph{Gromov-Hausdorff} sense to a
   point. Similarly, it is shown
   in~\cite[Proposition 3.1]{Tusch97} that a compact connected topological
  manifold $M$ is homeomorphic to an infra-solvmanifold if and only if
   $M$ admits a sequence of
  Riemannian metric $\{ g_{n}\}$ with uniformly bounded sectional curvature
  so that $(M, g_n)$ collapses in the \emph{Gromov-Hausdorff} sense to a \emph{flat orbifold}.
  \nn

  Note that any infra-solvmanifold has a canonical smooth structure which is induced from
  the simply connected solvable Lie group. On the other hand,
  a small cover or a real moment-angle manifold 
  may carry non-diffeomorphic smooth structures
  (e.g. the $7$-dimensional sphere). 
  So when we say a small cover or a real moment-angle manifold is an infra-solvmanifold, 
  a particular smooth structure should be chosen a priori. 
  In addition, since any small cover or real moment-angle manifold is equipped with a canonical 
  $\Z_2$-torus action, it is natural to consider \emph{equivariant smooth structures}, i.e.
   smooth structures which make the
  canonical $\Z_2$-torus action smooth.
  In the following, we show that
  equivariant smooth structures on small covers and real moment-angle manifolds
  always exist and 
 are unique up to equivariant diffeomorphisms.\nn
 
 To see the existence of equivariant smooth structures on a real moment-angle manifold 
 $\R \mathcal{Z}_{P^n}$, let us think of $\R \mathcal{Z}_{P^n}$ as 
 the pull-back by an embedding $i$ of $P^n$ into $\R^m_{\geq 0}$ from
 the standard $(\Z_2)^m$-action on $\R^m$, where $m$ is the number of facets of $P^n$
  (see the following commutative diagram).
   \[ \xymatrix{
           \R \mathcal{Z}_{P^n}  \ar[r]^{i_{\mathcal{Z}}} \ar[d] &  \R^m \ar[d]^{\mu}\\
        P^n \ar[r]^{i} &  \R^m_{\geq 0}              
           } \]
  Here $\mu(x_1,\cdots, x_m)=(x^2_1,\cdots, x^2_m)$, and 
  $i_{\mathcal{Z}}$ is a $(\Z_2)^m$-equivariant embedding. 
  The map $i$ is defined by the information of all the hyperplanes of $\R^n$ that bound $P^n$.
  Since the argument is completely parallel to the moment-angle manifolds case 
     in~\cite[$\S 3.1$]{BP-Panov-Ray07}, we leave it to the reader.
  It is not hard to see that $\R \mathcal{Z}_{P^n}$ embeds into $\R^m$ as the intersection
  of $m-n$ real quadrics 
   whose intersection is everywhere non-degenerate.
    So $\R \mathcal{Z}_{P^n}$ gets a $(\Z_2)^m$-equivariant smooth structure in this way.
      Moreover, any small cover over $P^n$ gets a $(\Z_2)^n$-equivariant smooth structure 
  as the smooth quotient
  of $\R \mathcal{Z}_{P^n}$ by a rank $m-n$ subgroup of $(\Z_2)^m$.\nn
  
 To show the uniqueness of equivariant smooth structures on small covers and real moment-angle
 manifolds, 
  we first recall some terminology and facts in the theory of $G$-normal systems developed by Davis in~\cite{Davis78}.
  It is shown in~\cite{Davis78} 
  that for any compact Lie group $G$, the diffeomorphism types of smooth $G$-manifolds 
   are in one-to-one correspondence with the isomorphism types of \emph{$G$-normal systems} 
 (see~\cite[Definition 4.1]{Davis78}).  
 In addition, the $G$-normal system associated to a smooth $G$-manifold
 $M$ determines a \emph{$B$-normal system} which further determines a
 \emph{local $G$-orbit space structure} on $M\slash G$ (see~\cite[p.335--336]{Davis13}),
  that is a collection of local
charts of $M\slash G$ so that the transition functions are stratified isomorphisms.
 It is shown in~\cite[Theorem 4.4]{Davis78} that there is a bijection between
isomorphism classes of local G-orbit spaces and isomorphism
classes of $B$-normal systems. 
  In particular when $M\slash G$ is a manifold with corners, a local $G$-orbit space structure
  on $M\slash G$ uniquely defines a smooth structure on $M\slash G$.
 These relations allow us to classify the diffeomorphism types of 
 smooth $G$-manifolds in terms of the smooth structures on the orbit spaces in some cases.
 For example, \cite{Wiem13} uses this strategy 
  to show that the $T$-equivariant smooth structure
 on a quasitoric manifold is unique up to equivariant diffeomorphism.
 \nn
  
  Indeed, we find that the argument in~\cite{Wiem13} also
 works for small covers and real moment-angle manifolds. 
 More specifically, suppose
  $M^n$ is an $n$-dimensional small cover over a simple polytope $P^n$ and $G=(\Z_2)^n$.   
    Let $\mathcal{T}_1$ and $\mathcal{T}_2$ be two $G$-equivariant 
  smooth structures on $M^n$.
   The $B$-normal systems on $P^n$ determined by $(M^n,\mathcal{T}_1)$ and
   $(M^n,\mathcal{T}_2)$
  correspond to two smooth structures on $P^n$.   
  Since it has been recently proved that 
 all smooth structures on a simple convex polytope are diffeomorphic 
  (see~\cite[Corollary 5.3]{Wiem13} or~\cite[Corollary 1.3]{Davis13}),
   the $B$-systems determined by
  $(M^n, \mathcal{T}_1)$ and $(M^n,\mathcal{T}_2)$ must be isomorphic. 
  Then using the same idea as the proof of~\cite[Theorem 5.6]{Wiem13}, we can
  inductively 
  construct an isomorphism between the $G$-normal systems of
 $(M^n, \mathcal{T}_1)$ and $(M^n,\mathcal{T}_2)$ according to the isomorphism
 of their $B$-normal systems. In fact, the construction
  in our case
  is much easier than that in~\cite{Wiem13} since
   our group $G=(\Z_2)^n$ has trivial homotopy groups, so there is no obstruction to
   extending an isomorphism 
    from the $k$-dimensional strata to $(k+1)$-dimensional
   strata of the $G$-normal systems for any $k\geq 0$. 
 So we obtain the following result which is parallel to~\cite[Corollary 5.7]{Wiem13}.\nn

 \begin{prop} \label{Prop:Equiv-Diffeo}
  If $\mathcal{T}_1$ and $\mathcal{T}_2$ are two equivariant 
  smooth structures on a small cover $M$, then
   $(M, \mathcal{T}_1)$ must be equivariantly diffeomorphic to $(M,\mathcal{T}_2)$.
  \end{prop}

 Similarly, we can show that any real moment-angle manifold $\R\mathcal{Z}_P$ has a
 unique equivariant smooth structure up to equivariant diffeomorphism.
 So in the rest of the paper, we always assume that a small cover or a real moment-angle manifold
carries the equivariant smooth structure.\nn

 \begin{cor} \label{Cor:Infra-Solv-2}
  Suppose $P^n$ is an $n$-dimensional simple convex polytope.
 \begin{itemize}
  \item[(i)] The real moment-angle manifold $\R\mathcal{Z}_{P^n}$
   of $P^n$ is homeomorphic to an infra-solvmanifold if and only if
    $\R\mathcal{Z}_{P^n}$ is diffeomorphic to the $n$-dimensional flat
    torus.\n

   \item[(ii)] A small cover $M^n$ over $P^n$
    is homeomorphic to an infra-solvmanifold if and only if $M^n$ is diffeomorphic
     to a real Bott manifold.
 \end{itemize}
 \end{cor}
 \begin{proof}
   By Corollary~\ref{Cor:Infra-Solv}, if a small cover $M^n$ over $P^n$
    is homeomorphic to an infra-solvmanifold, $P^n$ must be an
    $n$-dimensional cube. So $M^n$ is equivariantly homeomorphic to 
    a real Bott manifold $B_n$. Then since the canonical $(\Z_2)^n$-action
     on $B_n$ is smooth, Proposition~\ref{Prop:Equiv-Diffeo} implies that 
     $M^n$ must be equivariantly diffeomorphic to
     $B_n$. The proof of (i) is similar.     
  \end{proof}

 \n

 \noindent \textbf{Question:} Let $P$ and $Q$ be two simple convex polytopes.
   If $\R\mathcal{Z}_{P}$ and $\R \mathcal{Z}_{Q}$
 are homeomorphic or diffeomorphic, what can we conclude about the relationship between the combinatorial properties
  of $P$ and $Q$? \nn

  By Proposition~\ref{Prop:Asph-Flag}, if $\R\mathcal{Z}_{P}$ is homeomorphic to $\R \mathcal{Z}_{Q}$,
  then $P$ and $Q$ are either both flag or both
  non-flag. It would be interesting to see more answers to this question.
  \\

\section{Flag simple polytopes and Real Bott manifolds}
  In this section, we will get several different ways to describe a small
  cover that is homeomorphic to a real Bott manifold.
  As we know from the definition, any real Bott manifold of dimension $n$ is a small cover over an $n$-cube.
 It is clear that an $n$-cube is a flag simple polytope with $2n$ facets.
  Conversely, we can show the following.
\begin{prop} \label{flag}
 Let $P^n$ be a flag simple polytope of dimension $n$.  Then $P^n$
 has at least $2n$ facets, and if $P^n$ has exactly $2n$ facets,
  $P$ must be an $n$-cube.
\end{prop}
\begin{proof}
  This is probably known to many people. But since we do not know the literature, we
 shall give a proof here for the sake of completeness of the paper.\n

   Since $P^n$ is simple and of dimension $n$, there are $n$ facets in
  $P^n$ whose intersection is a vertex.  We denote them by
   $F_1,\dots,F_n$ and the vertex $\bigcap_{i=1}^n F_i$ by $v$.  For
  each $j\in [n]:=\{1,\dots,n\}$, the intersection
  $\bigcap_{i\not=j}F_i$ is an edge of $P^n$ which has $v$ as an
  endpoint.  Therefore, there is a unique facet of $P^n$, denoted
  $G_j$, such that $(\bigcap_{i\not=j}F_i)\cap G_j$ is the other
  endpoint of the edge $\bigcap_{i\not=j}F_i$ different from $v$. \n

  We claim that $G_j$'s must be mutually distinct.  Indeed, if
  $G_p=G_q$ for some $p\not=q$, this implies that any two of the $n+1$
  facets $F_1,\dots,F_n,G_p=G_q$ have non-empty intersection because
   $(\bigcap_{i\not=j}F_i)\cap G_j$ is non-empty for any $j$. Therefore
   the intersection of the $n+1$ facets must be non-empty since $P^n$
  is flag.  However, this is impossible because $P^n$ is simple and of
  dimension $n$.  Therefore $G_j$'s are mutually distinct and hence
  $P^n$ has at least $2n$ facets, proving the former statement of the
  proposition. \n

  Hereafter we assume that $P^n$ has exactly $2n$ facets.  Then the
  facets of $P^n$ are exactly $F_1,\dots,F_n, G_1,\dots,G_n$.  Since
  $\bigcap_{j=1}^nF_j$ and $(\bigcap_{i\not=j}F_i)\cap G_j$ are both
 non-empty, any two of the $n+1$ facets $F_1,\dots,F_n,G_j$ have
  non-empty intersection if $F_j\cap G_j\not=\emptyset$.  However,
  this is impossible by the same reason as above.  Therefore
  \begin{equation} \label{FGempty}
   \text{$F_j\cap G_j=\emptyset$ for any $j\in [n]$.}
  \end{equation}

  We shall prove that
\begin{equation} \label{FG}
 (\bigcap_{i\notin J}F_i) \cap (\bigcap_{j\in J}G_j)\not=\emptyset
  \quad\text{for any subset $J$ of $[n]$}
 \end{equation}
  by induction on the cardinality $|J|$ of $J$. Since
  $(\bigcap_{i\not=j}F_i)\cap G_j$ is a vertex of $P^n$ by the choice
  of $G_j$, \eqref{FG} holds when $|J|=1$.  Suppose that \eqref{FG}
  holds for $J$ with $|J|=k-1$.  Let $J$ be a subset of $[n]$ with
   $|J|=k$. Without loss of generality, we may assume that
   $J=\{1,2,\dots,k\}$. By the induction assumption, we have
    $(\bigcap_{i=k}^{n}F_i)\cap (\bigcap_{j=1}^{k-1}G_j)\not=\emptyset$.
    Since $P^n$ is simple, $(\bigcap_{i=k}^{n}F_i)\cap \bigcap_{j=1}^{k-1}G_j)$
    is a vertex of $P^n$, denoted $w$, and
  $(\bigcap_{i=k+1}^{n}F_i)\cap (\bigcap_{j=1}^{k-1}G_j)$ is an edge
  of $P^n$ which contains the vertex $w$.  Therefore, there is a
  unique facet $H$ of $P^n$ such that
  \begin{equation} \label{FGH}
   (\bigcap_{i=k+1}^{n}F_i)\cap (\bigcap_{j=1}^{k-1}G_j)\cap H \
   \text{is a vertex of $P^n$ different from $w$}.
  \end{equation}
  We claim that $H=G_k$.  In fact, since the intersection in~\eqref{FGH} is a vertex,
  $H$ must be either $F_p$ for $1\le p\le k$
    or $G_q$ for $k\le q\le n$.  However, the intersection in~\eqref{FGH} is empty unless $H=F_k$
    or $G_k$ by~\eqref{FGempty}.
   Moreover, $H\not=F_k$ because the intersection in~\eqref{FGH} is
  different from $w$.  Therefore we can conclude $H=G_k$ and this shows
  that~\eqref{FG} holds for $J$ with $|J|=k$, completing the induction
  step. \n

  Let $P^*$ be the simplicial polytope dual to $P^n$.
   Then the facts \eqref{FGempty} and \eqref{FG} show that the boundary
   complex $\partial P^*$ is isomorphic to the boundary complex of a
   crosspolytope $C$ of dimension $n$, which is isomorphic to the
  $n$-fold join of $S^0$.  Therefore the simplicial polytopes
   $P^*$ and $C$ are isomorphic combinatorially and hence so are
  their duals $P^n$ and $C^*$.  Since $C^*$ is an $n$-cube, this
  proves the latter statement of the proposition.
\end{proof}

\begin{rem}
   The above argument shows that if the geometrical realization of a
  flag simplicial complex $K$ is a pseudomanifold of dimension $n-1$,
  then the number of vertices of $K$ is at least $2n$; and if it is
  exactly $2n$, then $K$ is isomorphic to the boundary complex of the
  crosspolytope of dimension $n$.
\end{rem}

 Combining all our previous discussions, we get several
  descriptions of a small cover that is homeomorphic to a real Bott
  manifold as follows.\nn

\begin{thm} \label{thm:Test-Real-Bott}
 Suppose $M^n$ is a small cover over a simple convex polytope $P^n$ of dimension $n$.
 Let $b_1(M^n;\Z_2)$ denote the first Betti number of $M$ with $\Z_2$-coefficients. Then the
 following statements are equivalent.
 \begin{itemize}
   \item $M^n$ is homeomorphic to a real Bott manifold.\n
   \item $M^n$ is aspherical and $b_1(M^n ; \Z_2)\le n$. \n
   \item $P^n$ is flag and the number of facets of $P^n$ is $\leq 2n$. \n
   \item $P^n$ is an $n$-cube.
 \end{itemize}
\end{thm}
\begin{proof}
  Suppose $P^n$ has $r$ facets. It is known that $r=b_1(M^n;\Z_2) + n$ (see~\cite{DaJan91}).
  Then $b_1(M^n;\Z_2)\le n \Longleftrightarrow r\leq 2n$. In addition,
   $M$ is aspherical $\Longleftrightarrow P^n$ is flag
  by~\cite[Theorem 2.2.5]{DavJanScott98}. Then this proposition follows from Proposition~\ref{flag}.
\end{proof}

  \nn

  \nn

 \section{Riemannian metrics on small covers and real moment-angle manifolds}

     Geometric structures on small covers were first discussed in~\cite[Example 1.21]{DaJan91}
      for $3$-dimensional cases. Later, Davis-Januszkiewcz-Scott~\cite{DavJanScott98}
     systematically studied some piecewise Euclidean structures on small covers 
     called the \emph{natural piecewise Euclidean cubical metric}.
     A very nice result obtained in~\cite[Proposition 2.2.3]{DavJanScott98} says that
      the natural piecewise Euclidean cubical metric on a small cover over a simple polytope $P$
       is nonpositively curved if and only if $P$ is
        a flag polytope (this is also equivalent to
      saying that the small cover is aspherical).\nn

      In this section, we will study Riemannian metrics on small covers and
      real moment-angle manifolds in any dimension with
     certain conditions on the Ricci and sectional curvatures.
      By our discussion of the fundamental groups of small covers
     and real moment-angle manifolds in section 2, we will see that these
     curvature conditions put very strong restrictions on the
     topology of the corresponding small covers and real moment-angle
     manifolds and the combinatorics of the underlying simple
     polytopes.\nn

     A Riemannian manifold is called \emph{positively} (\emph{nonnegatively},
    \emph{nonpositively}, \emph{negatively}) \emph{curved} if its \emph{sectional curvature} is everywhere
    positive (nonnegative, nonpositive, negative). It is clear that a positively
    (nonnegatively, nonpositively, negatively) curved Riemannian manifold has positive
    (nonnegative, nonpositive, negative)
    \emph{Ricci curvature}. \nn

 \subsection{Positive curvature}

   By a classical theorem of Bonnet and Myers,
    any compact Riemannian manifold with positive Ricci curvature
     must have finite fundamental group (see Chapter 6 of~\cite{Peter06}).
   Then by Corollary~\ref{Cor:virtually-nilpotent} and Corollary~\ref{Cor:Finite-Fund-Group},
   we have the following corollary.
   \n

   \begin{cor} \label{Cor:Ricci}
   Let $P^n$ be an $n$-dimensional simple convex polytope.
    If the real moment-angle manifold $\R\mathcal{Z}_{P^n}$ admits a
    Riemannian metric with positive Ricci curvature, then $P^n$ must be $2$-neighborly
    and $\R\mathcal{Z}_{P^n}$ is simply connected.  Similarly, if a small cover $M^n$
    over $P^n$ admits a Riemannian metric with positive Ricci curvature, then $P^n$ must
    be $2$-neighborly and the fundamental group of $M^n$
   is isomorphic to $(\Z_2)^{r-n}$ where $r$ is the number of facets of $P^n$.
   \end{cor}
   \n

  The following is a well-known fact on positively curved Riemannian manifolds
 (see Chapter 6 of~\cite{Peter06}). \nn

  \begin{thm}[Synge 1936]
    Let $M$ be a compact Riemannian manifold with positive sectional curvature.
 \begin{itemize}
   \item[(i)] If $M$ is even-dimensional and orientable, then $M$
                    is simply connected.\n

  \item[(ii)] If $M$ is odd-dimensional, then $M$ is orientable.
\end{itemize}
  \end{thm}

  Notice that a small cover is never simply connected (this is an easy consequence
   of~\eqref{Equ:Fund-Group}). So by Synge's theorem, we can conclude the following.

   \begin{cor}
    If an even (odd) dimensional small cover admits a positively
   curved Riemannian metric, then it must be
   non-orientable (orientable).
   \end{cor}

 It is well known that
 a simply connected closed smooth $n$-manifold which admits a Riemannian metric with positive \emph{constant}
 sectional curvature is diffeomorphic to the standard sphere $S^n$ in $\R^{n+1}$.
\nn

 \begin{thm} \label{thm:Spherical}
  Let $P^n$ be an $n$-dimensional simple convex polytope.
  If the real moment-angle manifold $\R\mathcal{Z}_{P^n}$ (or a small cover $M^n$) of $P^n$ admits a
    Riemannian metric with positive constant sectional
    curvature, then $\R\mathcal{Z}_{P^n}$ (or $M^n$) is diffeomorphic to the standard sphere
    $S^n$ (or the real projective space $\R \mathbb{P}^n$).
 \end{thm}
 \begin{proof}
   If $\R\mathcal{Z}_{P^n}$ admits a
    Riemannian metric with positive constant sectional
    curvature, then it is simply connected by Corollary~\ref{Cor:Ricci}. So it is diffeomorphic to $S^n$.
    If a small cover $M^n$ over $P^n$ admits a Riemannian metric with positive constant sectional
    curvature, so does $\R\mathcal{Z}_{P^n}$. Then $\R\mathcal{Z}_{P^n}$ is diffeomorphic to $S^n$ and
    $\R\mathcal{Z}_{P^n}$ is a universal covering space of $M^n$.
    In addition, Corollary~\ref{Cor:Ricci} tells us that the fundamental group of $M^n$ is
     isomorphic to $(\Z_2)^{r-n}$ where $r$ is the number of facets of $P^n$.
    So $M^n$ is the quotient space of $\R\mathcal{Z}_{P^n} \cong S^n$ by a free
    $(\Z_2)^{r-n}$-action. But by the classical Smith's theory (see~\cite{Smith44}), we must have $r-n \leq 1$.
    So $P^n$ must be an $n$-dimensional simplex and then $M^n$ is the $n$-dimensional
     real projective space $\R \mathbb{P}^n$.
 \end{proof}

 The following geometric problem for small covers should be interesting to study.\nn

  \noindent  \textbf{Problem-1:} find out all
  the small covers (or real moment-angle manifolds) which admit Riemannian metrics with positive
  sectional or Ricci curvature in each dimension.\nn

   It is well known that the only $2$-neighborly simple polytopes
    in dimension $2$ and $3$ are the $2$-simplex and $3$-simplex.
    So by Corollary~\ref{Cor:Ricci},
    the only small covers in dimension $2$ and $3$
    that admit Riemannian metrics with positive Ricci curvature are $\R \mathbb{P}^2$ and $\R \mathbb{P}^3$.
  But in dimension $\geq 4$, the answer to Problem-1 is not so clear.
  In particular, it would be interesting to see if there exists a small cover that admits a positively curved Riemannian metric but that is not homeomorphic to a real projective space.\nn

     \subsection{Nonnegative curvature}

  By the Cheeger-Gromoll splitting theorem
  (see~\cite{ChgGrom71}), the fundamental group of any compact Riemannian
  manifold with nonnegative Ricci curvature is virtually abelian.
  This fact leads to the following description of the fundamental groups of real
  moment-angle manifolds and small covers that admit Riemannian metrics with nonnegative Ricci
  curvature.
\n

  \begin{prop}
  Let $P^n$ be an $n$-dimensional simple convex polytope with $r$ facets.
   If the real moment-angle manifold $\R\mathcal{Z}_{P^n}$ of $P^n$ admits a
   Riemannian metric with nonnegative
   Ricci curvature, then $\pi_1(\R\mathcal{Z}_{P^n})$ is isomorphic to $\Z^l$ for some
   $l\le r/2$.
  Similarly, if a small cover $M^n$ over $P^n$ admits a Riemannian metric
    with nonnegative Ricci curvature, then there is an exact sequence
 \[
       1 \longrightarrow  \Z^l  \longrightarrow \pi_1(M^n) \longrightarrow (\Z_2)^{r-n} \longrightarrow 1
 \]
  where $l\le r/2$.
  \end{prop}

  \begin{proof}
  If $\R\mathcal{Z}_{P^n}$ admits a Riemannian metric with nonnegative Ricci curvature,
  then $\pi_1(\R\mathcal{Z}_{P^n})$ is virtually abelian as remarked above.
  Therefore, the former statement in the proposition follows from
  Corollary~\ref{Cor:virtually-nilpotent}.
  If a small cover $M^n$ over $P^n$ admits a Riemannian metric with nonnegative Ricci curvature,
  then so does $\R\mathcal{Z}_{P^n}$ because $\R\mathcal{Z}_{P^n}$ is a finite cover of $M^n$.
  Therefore, the latter statement in the proposition follows from \eqref{Equ:moment-small}
  and the former statement.
  \end{proof}
  \n

  \begin{exam} \label{thm:NN-Curv}
   The real moment-angle manifold over a product of simplices
   $\Delta^{n_1}\times\cdots \times \Delta^{n_m}$ is a product of
   standard spheres $S^{n_1}\times \cdots\times S^{n_m}=:S$.
   In the product, let each sphere $S^{n_i} \subset \R^{n_i +1}$
  be equipped with the standard Riemannian metric
    whose isometry group is the orthogonal group $O(n_i+1,\R)$.
  Then $S$ is a nonnegatively curved Riemannian manifold with respect
  to the product metric.\n

 A generalized real Bott manifold $B_m$ discussed in Example~\ref{Exam:Bott} is a small cover over
 $\Delta^{n_1}\times\cdots \times \Delta^{n_m}$.  It is known that $B_m$ is the quotient of $S$
 by a free $(\Z_2)^m$-action on it (see \cite[Proposition 6.2]{SuyMasDong10})
 and one can easily see that the $(\Z_2)^m$-action
 preserves the product metric on $S$. So the quotient space of
  this free $(\Z_2)^m$-action, that is $B_m$,
   inherits a nonnegatively curved Riemannian metric from $S$.
  Noitce that the product metric on $S$ has positive Ricci curvature if and only if $n_1,\cdots,
   n_m>1$. So $B_m$ admits a Riemannian metric with positive Ricci curvature
  if and only if $n_1,\cdots,
   n_m>1$.\end{exam}
   \n

    \begin{rem}
       It is possible that $B_m$ is the total space of several different
       generalized real Bott towers.
        Therefore, we will get several different nonnegatively curved
        Riemannian metrics on $B_m$ which may not be isometric.
    \end{rem}
    \n

 \noindent  \textbf{Problem-2:} find out all
  the small covers (or real moment-angle manifolds) which admit Riemannian metrics with nonnegative
  sectional or Ricci curvature in each dimension.\nn

    In dimension $2$ and dimension $3$, the small covers which
  admit Riemannian metrics with nonnegative Ricci curvature are
  exactly the generalized real Bott manifolds. This follows from the
  classification of $3$-dimensional compact Riemannian manifolds with nonnegative Ricci curvature
  in~\cite{Hamil86}. So by Example~\ref{thm:NN-Curv}, we can conclude that
  all the nonnegatively curved small covers (real moment-angle manifolds) in dimension $2$ and dimension $3$ are exactly
 generalized real Bott manifolds (product of spheres). But in dimension $\geq 4$, the answer
 to Problem-2 is not so easy. In particular, it is interesting to see
 if there exists any small cover which admits a nonnegatively curved Riemannian
 metric but not homeomorphic to any generalized real Bott
 manifold.\nn

 \subsection{Non-positive and negative curvature}\ \n

 If a real moment-angle manifold
 $\R \mathcal{Z}_P$ (or a small cover $M$) of a simple polytope $P$ admits a
 non-positive sectional curvature, the
 Cartan-Hadamard theorem implies that its universal covering space
 is diffeomorphic to an Euclidean space. Hence $\R \mathcal{Z}_P$
 is aspherical and so $P$ is a flag polytope (by~\cite[Theorem 2.2.5]{DavJanScott98}).
 Conversely, if $P$ is a flag polytope,~\cite[Theorem 2.2.3]{DavJanScott98}
  tells us that $\R \mathcal{Z}_P$ admits a piecewise Euclidean metric which is nonpositively curved
  (as a metric space).
 \nn

 \textbf{Problem-3:} For a flag simple polytope $P$, does there exist a nonpositively curved
  Riemannian metric on $\R \mathcal{Z}_P$ (or any small cover over $P$)? \nn

  In addition, a theorem due to Preissmann~\cite{Press43} says that any abelian subgroup of the fundamental group of a
 negatively curved compact Riemannian manifold is
 infinite cyclic (also see~\cite[section 9.3]{Burago2001}). Using this fact, we can easily
 show the following.\nn

 \begin{prop}\label{Prop:Negative-Curv}
  If the real moment-angle manifold $\R \mathcal{Z}_{P}$ (or a small cover)
  of a simple polytope $P$ admits a
   negatively curved Riemannian metric, then no $2$-face of $P$ can be a $3$-gon or a $4$-gon.
 \end{prop}
 \begin{proof}
 If $\R \mathcal{Z}_{P}$ admits a negatively curved Riemannian
   metric, then it follows from Cartan-Hadamard theorem that  $\R \mathcal{Z}_{P}$ is aspherical.
   So the simple polytope $P$ is flag by Proposition~\ref{Prop:Asph-Flag}.
   Then by Lemma~\ref{Lem:Face-Flag}, any $2$-face $f$ of $P$ must be flag.
  So $f$ can not be a $3$-gon.
   In addition, since $P$ is flag, the proof of
   Proposition~\ref{prop:flag-polytop-solv}
   implies that there is an injective group homomorphism from $\pi_1(\R \mathcal{Z}_{f})$
   into $\pi_1(\R \mathcal{Z}_{P})$. This implies that $f$ can not be a
   $4$-gon, because otherwise $\pi_1(\R \mathcal{Z}_{f}) \cong \Z \oplus \Z$
   and so $\pi_1(\R \mathcal{Z}_{P})$ contains an abelian subgroup $\Z \oplus \Z$
   which contradicts the Preissmann's theorem mentioned above.
  Finally, if a small cover over $P$ admits a negatively curved Riemannian
   metric, then so does $\R \mathcal{Z}_{P}$.
 \end{proof}

 \n

  \begin{prop}\label{Prop:Hyperbolic}
   The real moment-angle manifold $\R \mathcal{Z}_{P}$ (or a small cover)
   over a $3$-dimensional simple convex polytope $P$ admits a
  hyperbolic structure (i.e. Riemannian metric with negative constant sectional curvature) if and
  only if $P$ has no prismatic $3$-circuits or $4$-circuits.
  \end{prop}
    A \emph{prismatic $k$-circuit} on $P$ is a simple closed curve $\Gamma$ formed
    of $k$-edges of $P^*$ (the dual simplicial polytope of $P$) so
    that all of the endpoints of the edges of $P$ intersected by
    $\Gamma$ are distinct. The famous \emph{Andreev theorem} (see~\cite{Andreev70, Pog67, RoHubDunb07})
      says that $P$ can be realized as a right-angled polytope in
    the hyperbolic $3$-space if and only if $P$ has no prismatic $3$-circuits or
    $4$-circuits.
  \begin{proof}
       If $P$ has no prismatic $3$-circuits or $4$-circuits,
       then $P$ can be realized as a right-angled polytope in
    the hyperbolic $3$-space. So
      $\R \mathcal{Z}_{P}$ (or any small cover over $P$)
      will get a hyperbolic structure which is invariant under
    the canonical $\Z_2$-torus actions on it. \n

    Conversely, assume $\R \mathcal{Z}_{P}$ (or a small cover
   over $P$) admits a hyperbolic structure.
   It was shown in~\cite{DinkLeeb09} that
    any smooth action by a finite
    group on a closed hyperbolic $3$-manifold is smoothly conjugate to an
   isometric action. So we can assume that the
    canonical $\Z_2$-torus action on $\R \mathcal{Z}_{P}$ (or the small cover)
  is isometric with respect to the hyperbolic structure.
   Then since the $\Z_2$-torus action is locally standard,
   the orbit space $P$ has a natural structure of a right-angled hyperbolic
   polyhedron. Then Andreev's theorem implies
    that $P$ has no prismatic $3$-circuits or $4$-circuits.
  \end{proof}
  \nn

 \begin{rem}
    There is no analogue of Andreev theorem in any dimension $\geq 4$.
    So it is not clear how to judge the existence of hyperbolic structures on
    real moment-angle manifolds or small covers in general.\\
  \end{rem}

  \section*{Acknowledgements}
   The authors want to thank Y.~Kamishima and J.~B.~Lee for helpful
   comments on infra-solvmanifolds, and thank Z.~L\"u and J.~M.~Ma for comments on
   Proposition~\ref{Prop:Hyperbolic}.
   \\

\end{document}